\documentclass[psamsfonts]{amsart}

\usepackage{amssymb,amsfonts}
\usepackage[all,arc]{xy}
\usepackage{enumerate}
\usepackage{mathrsfs}
\usepackage{amsmath}

\newtheorem{thm}{Theorem}[section]

\newtheorem{prop}[thm]{Proposition}

\theoremstyle{definition}
\newtheorem{defn}[thm]{Definition}

\newtheorem{exmp}[thm]{Example}

\newtheorem{notn}[thm]{Notation}

\theoremstyle{remark}
\newtheorem{rem}[thm]{Remark}

\makeatletter
\let\c@equation\c@thm
\makeatother
\numberwithin{equation}{section}

\title{Relation Functions Evaluated from Unique Coefficient Patterns}

\author{Alperen Sirin}
\address{Department of Mathematics\\
                University of Rochester\\
                }
\email{saidsirin@gmail.com}

\begin{document}

\begin{abstract}
In this paper, we study polynomials of the form $f(x)=(x^n+x^{n-1}+...+1)^l$ for $l=1,2,3,4$ to generate a pattern titled "unique coefficient pattern". Namely, we analyze each unique coefficient patterns of $f(x)$ and generate functions titled "relation functions". The approach that we follow will allow us to evaluate desired coefficients for such polynomial expansions by simply using these relation functions.
\end{abstract}

\maketitle

\section{Introduction}
For centuries, many mathematicians have studied the Pascal Triangle and generalized different coefficient patterns such as the trinomial triangle or the quadrinomial triangle. Consequently, first the binomial theorem, and then a more generalized version called multinomial theorem was established to find any coefficient of any polynomial expansions. In this paper however, we will demonstrate a new pattern that occurs from the unique coefficient of such polynomials.
\begin{defn}
A coefficient pattern is the set of coefficients from the expansion of $(x^n+x^{n-1}+...+1)^l$ for any positive integer $l$ and $n$.
\end{defn}
\subsection{Unique Coefficient Pattern}
\begin{defn}
Unique Coefficient Pattern, denoted as $U_r(k)$ is the set of coefficients of which a certain coefficient pattern have but preceding patterns dont.
\end{defn}
 For example, consider $f(x)=(x^n+...+1)^2$. Then, $U_1(2)=1$, $U_2(2)=2$,..., $U_r(2)=r$. Listing the set of $U_r(2)$, we have,
\[\begin{array}{ccccccccccc}
$n=1$&    &    &    &    &  1\\\noalign{\smallskip\smallskip}
$n=2$&    &    &    &    &  2\\\noalign{\smallskip\smallskip}
$n=3$&    &    &    &    &  3\\\noalign{\smallskip\smallskip}
\vdots \cr
\end{array}\] 
\[\begin{array}{ccccccccccc}
$n=r$&    &    &    &    & $r$\\\noalign{\smallskip\smallskip}
\end{array}\]
Similarly, lets consider $f(x)=(x^n+...+1)^3$. Then $U_1(3)=1$, $U_2(3)=(3,3)$, $U_3(3)=(6,7,6)$,... and so on. Listing the set of $U_r(3)$, we have the following pattern,
\[\begin{array}{ccccccccccc}
&    &    &    &    &  1\\\noalign{\smallskip\smallskip}
&    &    &    &  3  &  &  3\\\noalign{\smallskip\smallskip}
&    &    &  6 &    &  7 &    &  6\\\noalign{\smallskip\smallskip}
&    &  10 &    &  12 &    &  12 &    &  10\\\noalign{\smallskip\smallskip}
&  15 &    &  18 &    &  19 &    &  18 &    &  15
\dots \cr
\end{array}\]
Lastly, lets consider $f(x)=(x^n+...+1)^4$. Then, $U_1(4)=1$, $U_2(4)=(4,6,4)$, $U_3(4)=(10,16,19,16,10)$, ... and so on. In the same manner, listing the set of $U_r(4)$, we have,
\[\begin{array}{ccccccccccc}
&    &    &    &    &  1\\\noalign{\smallskip\smallskip}
&    &    &    &  4  & 6 &  4\\\noalign{\smallskip\smallskip}
&    &    &  10 &  16  & 19  & 16  &  10\\\noalign{\smallskip\smallskip}
&    &  20 & 31  &  40 & 44   &  40 &  31  & 20\\\noalign{\smallskip\smallskip}
\vdots \cr
\end{array}\]

\begin{notn}
Let $i_n=$specific term $n$, $i_{n_k}=$specific coefficient in term $n$ where $k=$row number, $g_l(i_{n_k})=$specific coefficient for the $l$'th power of which $n$ and $k$ correspond to.
\end{notn}
\begin{defn}
A relation function is namely $g_l(i_{n_k})$, which has a domain and range consisting of $Z_+^*$
\end{defn}

\section{Methods/Derivation}
\begin{prop}
The relation function to evaluate any unique coefficient for $l=2$ is
\[\ g_2(i_n)=n\]
\end{prop}
\begin{proof}
Note that this is very straightforward. The set of $U_r(2)$ are consecutive positive integers.
\end{proof}

\begin{prop}
The relation function to evaluate the first and last unique coefficients for $l=3$ is
\[\ g_3(i_{n_(1,n)})=\frac{n(n+1)}{2}\\\]
\end{prop}
\begin{proof}
Note that $g_3(i_{1_1})=1$, $g_3(i_{2_1})=3$, $g_3(i_{3_1})=6$. So in general we see that, $g_3(i_{n_1})=\sum\limits_{j=1}^{n} j=\frac{n(n+1)}{2}\\$
\end{proof}

\begin{prop}
The relation function to evaluate the unique coefficients for $l=3$ when ($k\neq1,n$) is
\[\ g_3(i_{n_k})=\frac{2k(n+1)-2k^2+n(n-1)}{2}\\\]
\end{prop}
\begin{proof}
Note that
\[\ g_3(i_{n_2})=g_3(i_{{n-1}_1})+2i_{n-1}\]
Similarly,
\[\ g_3(i_{n_3})=g_3(i_{{n-2}_1})+2i_{n-2}+2i_{n-2}+1\]
Also,
\[\ g_3(i_{n_4})=g_3(i_{{n-3}_1})+2i_{n-3}+2i_{n-3}+1+2i_{n-3}+2\]
Hence
\[\ g_3(i_{n_k})=g_3(i_{{n-k+1}_1})+\sum\limits_{m=1}^{k-2} 2i_{n-k+1}+m\]
=
\[\ \frac{2k(n+1)-2k^2+n(n-1)}{2}\\\]
\end{proof}

\begin{prop}
The relation function to evaluate the first and last unique coefficients for $l=4$ is
\[\ g_4(i_{n_(1,n)})= \frac{(n^2+n)(n+2)}{6}\\ \]
\end{prop}
\begin{proof}
Note that $g_4(i_{1_1})=1$, $g_4(i_{2_1})=4$, $g_4(i_{3_1})=10$, and so on. In general we see that, $g_4(i_{n_1})=\sum\limits_{j=1}^{n} \frac{j(j+1)}{2}= \frac{(n^2+n)(n+2)}{6}\\$. By symmetry, the same holds for $g_4(i_{n_n})$.
\end{proof}

\begin{prop}
The relation function to evaluate the unique coefficients for $l=4$ conditioned that $n\neq1,2,3,...,(k-1)$ is
\[\ g_4(i_{n_k})= \frac{(k-1)[((n^2+n)+nk)-k(k+1)]}{2}\\ \]
\end{prop}
\begin{proof}
Note that
\[\ g_4(i_{2_2})=g_4(i_{2_1})+i_2\],
\[\ g_4(i_{3_2})=g_4(i_{3_1})+i_2+4\],
\[\ g_4(i_{4_2})=g_4(i_{4_1})+i_2+4+5\]
and so on. Thus, in general
\[\ g_4(i_{n_2})=g_4(i_{n_1})+\sum\limits_{j=1}^{n} (j+1) - 3\]=
\[\sum\limits_{j=1}^{n} \frac{j(j+1)}{2}\\+\sum\limits_{j=1}^{n} (j+1) - 3\] for $n\neq1$.
Similarly,
\[\ g_4(i_{3_3})=g_4(i_{3_2})+i_3\],
\[\ g_4(i_{4_3})=g_4(i_{4_2})+i_3+6\],
\[\ g_4(i_{5_3})=g_4(i_{5_2})+i_3+6+7\]
Thus,
\[\ g_4(i_{n_3})=g(i_{n_2})+\sum\limits_{j=1}^{n} (j+2) - (4+5)\]=
\[\sum\limits_{j=1}^{n} \frac{j(j+1)}{2}\\+\sum\limits_{j=1}^{n} (j+1) - 3+\sum\limits_{j=1}^{n} (j+2) - (4+5)\] for $n\neq1,2$
Also observe that
\[\ g_4(i_{4_4})=g_4(i_{4_3})+i_4\],
\[\ g_4(i_{5_4})=g_4(i_{5_3})+i_4+8\],
\[\ g_4(i_{6_4})=g_4(i_{6_3})+i_4+8+9\]
Thus,
\[\ g_4(i_{n_4})=g_4(i_{n_3})+\sum\limits_{j=1}^{n} (j+3) - (5+6+7)\]=
\[\ \sum\limits_{j=1}^{n} \frac{j(j+1)}{2}\\+\sum\limits_{j=1}^{n} (j+1) - 3+\sum\limits_{j=1}^{n} (j+2) - (4+5)+\sum\limits_{j=1}^{n} (j+3) - (5+6+7)\] for $n\neq1,2,3$
In general,
\[\ g_4(i_{n_k})=\sum\limits_{j=1}^{n} \frac{j(j+1)}{2}\\+\frac{(k-1)(n^2+n)}{2}\\+n(\sum\limits_{j=1}^{k-1} j\\) - \sum_{k=2}^{k}(\sum_{j=k+1}^{2k-1} j)\\\]=
\[\ g_4(i_{n_k})= \frac{(k-1)[((n^2+n)+nk)-k(k+1)]}{2}\\ \]
\end{proof}
\begin{rem}
Note that the restriction $n\neq1,2,3,...,(k-1)$ does not prevent us from finding the unique coefficients for these values of $n$ due to the symmetry in the set of $U_r(k)$.
\end{rem}

\section{Applications}
The applications of relation functions could be seen in finding coefficients of polynomial expansions. Note that one can evaluate any desired coefficient of any polynomial expansion by using the multinomial theorem. However, the relation functions that are derived in this paper could also be used.
\begin{exmp}
Evaluate the coefficients of $(x^2+x+1)^3$ using relation functions.
\end{exmp}
Note that
\[\ g_3(i_{1_1})=\frac{(i_1)(i_2)}{2}\\=1\],
\[\ g_3(i_{2_1})=\frac{2(2+1)-2^2+2(2-1)}{2}\\=3\],
\[\ g_3(i_{3_1})=\frac{2(3+1)-2^2+3(3-1)}{2}\\=6\],
\[\ g_3(i_{3_2})=\frac{4(3+1)-4^2+3(3-1)}{2}\\=7\].

Thus, by symmetry, it is sufficient to find the first 4 coefficients. As a result, the coefficients of this expansion will be $1, 3, 6, 7, 6, 3, 1$
\begin{exmp}
Consider the polynomial,
 \[\ (2x+1)(x^4+x^3+x^2+x+1)^3\]
Find the coefficient of $x^9$ using relation functions and your intuition
\end{exmp}
Since we are multiplying by $(2x+1)$, we only consider the coefficients of $x^9$ and $x^8$ in the expansion of the right hand side. Therefore we want to evaluate $g_3(i_{4_1})$ and $g_3(i_{5_1})$. So,
\[\ g_3(i_{4_1})=\frac{2(4+1)-2^2+4(4-1)}{2}\\=10\]
and
\[\ g_3(i_{5_1})=\frac{2(5+1)-2^2+5(5-1)}{2}\\=15\]
Thus, the coefficient of $x^9$ is $10+30=40$.

\section{Conclusion}
In this paper, we demonstrated a pattern that occurs from the set of unique coefficients for polynomials in the form of $f(x)=(x^n+...+1)^l$. However, we only considered the cases when $l=2,3,4$. It should be noted that the pattern we generated is different from pascal's triangle and its generalizations. Moreover, as a continuation of this paper, higher powers of $l$ could also be observed in a similar manner to obtain one single unifying equation for $g_l(i_{n_k})$. Finally as a result of this paper, along with a unique pattern, a new method to evaluate the coefficients of such polynomials are generated independently from binomial and multinomial theorem. 
\section{Acknowledgements}
I would like to present my appreciation and thank my adviser Professor Richards Geordie for his evaluations and suggestions. Moreover, I would like to thank my peer mentor Eyub Yegen for being a source of motivation and inspiration. Most importantly, I would like to thank my father Seyit Sirin for deeply supporting me through this whole process.


\begin{thebibliography}{9}

\bibitem{generalizedpascaltriangle}
Wong, C. K., and T. W. Maddocks.
"GENERALIZED PASCALS TRIANGLE."
 Fibonacci Quarterly 13.2 (1975): 134-136.

\bibitem{Hoggatt}
Hoggatt Jr, Vo E., and Marjorie Bicknell.
"Diagonal Sums of Generalized Pascal Triangles."
The Fibonacci Quarterly 7.4 (1969): 341-358. 

\bibitem{trinomial}
Chappell, James, and Thomas J. Osler.
"The trinomial triangle."
College Mathematics Journal (1999): 141-142.

\end{thebibliography}
\end{document}